\documentclass[11pt, a4paper, reqno]{amsart}

\usepackage{amsfonts,amsmath,amsthm}
\usepackage{amssymb}
\usepackage{latexsym}

\usepackage[english]{babel}

\theoremstyle{plain}
\newtheorem{theorem}{Theorem}[section]
\newtheorem{lemma}[theorem]{Lemma}

\newtheorem{cor}[theorem]{Corollary}
\newtheorem{prop}[theorem]{Proposition}

\theoremstyle{definition}
\newtheorem{definition}[theorem]{Definition}

\theoremstyle{remark}

\newcommand{\eps}{\varepsilon}
\newcommand{\N}{{\mathbb N}}

\newcommand{\R}{{\mathbb R}}
\newcommand{\ifff}{if and only if }

\newcommand{\Id}{\mathrm{Id}}
\newcommand{\C}{\mathbb{C}}
\newcommand{\K}{\mathbb{K}}

\newcommand{\dopu}{{:}\allowbreak\ }

\newcommand{\loglike}[1]{\mathop{\rm #1}\nolimits}

\newcommand{\DPr}{\loglike{DPr}}

\newcommand{\sign}{\loglike{sign}}

\newcommand{\lin}{\loglike{lin}}

\newcommand{\re}{\loglike{Re}}

\newcommand{\Norm}{|\mkern-2mu|\mkern-2mu|}
\newcounter{abc}   % Counter f\"{u}r statements-environment wird deklariert
\newcounter{iiiii} % Counter f\"{u}r aequivalenz-environment wird deklariert

\newenvironment{aequivalenz}
{\setcounter{iiiii}{0}
\begin{list}%
{{\rm (\roman{iiiii})}}%  Falls die items nicht angegeben sind: i)u.s.w.
{\usecounter{iiiii}
\parsep=0pt plus 1pt
\topsep=1pt plus 2pt minus 1pt
\itemsep=1pt plus 2pt minus 1pt
\leftmargin=3\baselineskip \labelsep=.6\baselineskip
\labelwidth=2.4\baselineskip
\rightmargin 0pt}%
}%               Das war das zweite Argument von "newenvironment"
{\end{list}}

\newenvironment{statements}%
{\setcounter{abc}{0}
\begin{list}%
{{\rm (\alph{abc})}}%  Falls die items nicht angegeben sind: (a) u.s.w.
{\usecounter{abc}
\parsep=0pt plus 1pt
\topsep=1pt plus 2pt minus 1pt
\itemsep=1pt plus 2pt minus 1pt
\leftmargin=3\baselineskip \labelsep=.6\baselineskip
\labelwidth=2.4\baselineskip
\rightmargin 0pt}%
}%               Das war das zweite Argument von "newenvironment"
{\end{list}}

{\begin{list}%
{$\bullet$}%  Falls die items nicht angegeben sind: Punkte
{\leftmargin=3\baselineskip
\labelsep=\baselineskip \labelwidth=2.5\baselineskip
\parsep=0pt plus 1pt
\topsep=1pt plus 2pt minus 1pt
\itemsep=1pt plus 2pt minus 1pt
\rightmargin 0pt}%
}%               Das war das zweite Argument von "newenvironment"
{\end{list}}

\newcommand{\bea}{\begin{eqnarray*}}
\newcommand{\eea}{\end{eqnarray*}}
\newcommand{\beq}{\begin{equation}}
\newcommand{\eeq}{\end{equation}}
\newcommand{\begsta}{\begin{statements}}
\def\endsta{\end{statements}}
\newcommand{\begaeq}{\begin{aequivalenz}}
\def\endaeq{\end{aequivalenz}}

\newcommand{\iy}{\infty}
\def\DP{Daugavet property}

\newcommand{\BS}{Banach space}

\newcommand{\TFAE}{the following conditions are equivalent: }

\numberwithin{equation}{section}

\begin{document}

\title%[Inner measure of non-compactness of the sphere]
[Thickness, $\ell_1$-types, and the almost Daugavet property]
{Thickness of the unit sphere,
$\ell_1$-types, and the almost Daugavet property}

\author{Vladimir Kadets, Varvara Shepelska and Dirk Werner}

\date{February 5, 2009}



 \thanks{The first author was partially
 supported by a grant from the {\it Alexander-von-Humboldt
Foundation} and by Junta de Andaluc\'{\i}a grant P06-FQM-01438;
the second author was partially supported by the N.~I.~Akhiezer
foundation.}

\address{Department of Mechanics and Mathematics,
Kharkov National University,
 pl.~Svobody~4,  61077~Kharkov, Ukraine}
\email{vova1kadets@yahoo.com; shepelskaya@yahoo.com}

\address{Department of Mathematics, Freie Universit\"at Berlin,
Arnimallee~6, \qquad {}\linebreak D-14\,195~Berlin, Germany}
\email{werner@math.fu-berlin.de}

\begin{abstract}
We study those Banach spaces $X$ for which  $S_X$ does not admit a finite
$\eps$-net consisting of elements of $S_X$ for any $\eps < 2$. We give
characterisations of this class of spaces in terms of
$\ell_1$-type sequences and in terms of  the almost Daugavet
property. The main result of the paper is: a separable Banach space $X$
is isomorphic to a space from this class if and only if $X$ contains an
isomorphic copy of $\ell_1$.
\end{abstract}

\maketitle

\thispagestyle{empty}

\section{Introduction}
For a Banach space $X$, R.~Whitley \cite{Whitley} introduced the
following parameter, called \textit{thickness},
which is essentially the inner measure of
non-compactness of the unit sphere $S_X$:
$$
T(X) = \inf \{\eps > 0\dopu  \text{there exists a finite }
\eps\text{-net for }  S_X \text{ in } S_X \},
$$
or equivalently, $T(X)$ is the infimum of those $\varepsilon$ such that the unit
sphere of $X$ can be covered by a finite number of balls with
radius $\varepsilon$ and centres in $S_X$.
He  showed  in the infinite dimensional case that $1 \le T(X) \le 2$,
and in particular that  $T(C(K))=1$ if $K$ has isolated points and
$T(C(K))=2$ if not.

In this paper we concentrate on the spaces with
$T(X)=2$. Our main results are the following; $B_X$ denotes the closed
unit ball of~$X$.

\begin{theorem} \label{th0.1}
For a separable \BS \ $X$ \TFAE
\begsta
\item $T(X)=2$;
 \item there is a sequence $(e_n) \subset B_X$ such
that for every $x \in X$
$$
\lim_{n\to\infty} \|x+e_n\| =\|x\|+1;
$$
 \item there is a norming subspace $Y \subset X^*$ such that the equation
 \beq \label{0-eq1}
  \|\Id+T\|=1+\|T\|
 \eeq
holds true for every rank-one operator $T\dopu X \to X$ of the
form $T = y^*\otimes x$, where $x \in X$ and $y^* \in Y$.
\endsta
\end{theorem}

\begin{theorem} \label{th0.2}
A separable \BS \ $X$ can be equivalently renormed to have thickness
$T(X)=2$ if and only if $X$ contains an isomorphic copy of $\ell_1$.
\end{theorem}

We mention that it has been proved in \cite{BCP}
that a space with thickness
$T(X)=2$ contains a copy of~$\ell_1$. 

Recall that a subspace $Y \subset X^*$ is said  to be norming (or
1-norming) if for every $x \in X$
$$
\sup_{y^* \in S_Y} |y^*(x)| = \|x\|.
$$
$Y$ is norming \ifff $S_Y$ is weak$^*$ dense in $B_{X^*}$.

Condition~(b) of Theorem~\ref{th0.1} links our investigations to
the theory of types \cite{KriMau}. Recall that a type on a
separable Banach space $X$ is a function of the form
$$
\tau(x)= \lim_{n\to \infty} \|x+e_n\|
$$
for some bounded sequence $(e_n)$. In \cite{KriMau} the notion of an
$\ell_1$-type is defined by means of convolution of types; a special
instance of this is a type generated by a sequence $(e_n)$ satisfying
\begin{equation} \label{eq-type}
\tau(x)= \lim_{n\to \infty} \|x+e_n\| = \|x\| + 1.
\end{equation}
To simplify notation let us call
a type like this a \textit{canonical $\ell_1$-type} and
a sequence $(e_n)\subset B_X$ satisfying (\ref{eq-type}) a
\textit{canonical $\ell_1$-type sequence}.

Condition~(c) links our investigations to the theory of
\BS s with the \textit{\DP} introduced in \cite{KadSSW} and developed further
for instance in the papers
\cite{BKSW} \cite{IvKadWer}, \cite{KadKalWer}, \cite{KW}; see also the
survey \cite{Wer}.
We will say that a Banach space
$X$  has the \textit{\DP\ with respect to $Y$ $(X\in \DPr(Y))$} if
the \textit{Daugavet equation} (\ref{0-eq1}) holds true for every
rank-one operator $T\dopu X \to X$ of the form $T = y^*\otimes x$,
where $x \in X$ and $y^* \in Y$, and it has the \textit{almost \DP} or
is an \textit{almost Daugavet space} if it has $\DPr(Y)$ for some
norming subspace $Y\subset X^*$.
This definition is a generalization (introduced in
\cite{KadShepW})  of the by now well-known \DP\ of \cite{KadSSW},
which is $\DPr(Y)$ with $Y=X^*$.

In this language Theorem~\ref{th0.2} says, by Theorem~\ref{th0.1},
that a separable Banach space can be renormed to have the almost \DP\
if and only if it contains a copy of $\ell_1$.

In Section~2 we present a characterisation of almost Daugavet spaces
in terms of $\ell_1$-sequences of the dual. The proofs of
Theorems~\ref{th0.1} and~\ref{th0.2} will be given in Sections~3
and~4.

The following lemma
is the main technical prerequisite that we use; it
is the analogue of \cite[Lemma~2.2]{KadSSW}.
Up to part~(v) it was proved in \cite{KadShepW}; however, (v)
follows along the same lines.
By a slice of $B_X$ we mean a set of the form
$$
S(y^*,\eps) = \{x\in B_X\dopu \re y^*(x) \ge 1-\eps\}
$$
for some $y^*\in S_{X^*}$ and some $\eps>0$, and a weak$^*$ slice
$S(y,\eps)$ of the dual ball $B_{X^{*}}$ is a particular case of
slice, generated by element $y \in S_X \subset X^{**}$.

\begin{lemma}\label{1-th2}
If $Y$ is a norming subspace of $X^*$, then the following
assertions are equivalent.
 \begaeq
 \item
$X$  has the \DP\ with respect to $Y$.
 \item
For every $x \in S_{X}$, for every $\eps > 0$, and for every
$y^*\in S_{Y}$ there is some $y \in S(y^*,\eps)$ such that
\begin{equation} \label{1-eq2}
\|x+ y\|\ge 2-\eps.
\end{equation}
 \item
For every $x \in S_{X}$, for every $\eps > 0$, and for every
$y^*\in S_{Y}$ there is a slice $S(y_1^*,\eps_1) \subset
S(y^*,\eps)$ with $y_1^*\in S_{Y}$ such that {\rm(\ref{1-eq2})}
holds for every $y \in S(y^*,\eps_1)$.
 \item
For every $x^*\in S_{Y}$, for every $\eps > 0$, and for every
weak$^*$
slice $S(x,\eps)$ of the dual ball $B_{X^{*}}$ there is some
$y^*\in S(x,\eps)$ such that $\|x^{*}+ y^*\|\ge 2-\eps$.
 \item
For every $x^*\in S_{Y}$, for every $\eps > 0$, and for every
weak$^*$ slice $S(x,\eps)$ of the dual ball $B_{X^{*}}$ there is
another weak$^*$ slice $S(x_1,\eps_1) \subset S(x,\eps)$ such that
$\|x^{*}+ y^*\|\ge 2-\eps$ for every $y^*\in S(x_1,\eps_1)$.
 \endaeq
\end{lemma}

We would like to thank P.~Papini for providing us with useful
references and W.B. Johnson, E. Odell and Th.~Schlumprecht for
helpful comments.

\section{A characterisation of almost Daugavet spaces
by means of  $\ell_{1}$-sequences in the dual}

For the sake of easy notation we introduce two definitions.

\begin{definition} \label{2-def1}
Let $E$ be subspace of a \BS\ $F$ and $\eps > 0$. An element $e \in
B_F$ is said to be $(\eps, 1)$-{\it orthogonal}\/ to $E$ if for
every $x \in E$ and $t \in \R$
 \beq\label{2-eq1}
 \|x+te\|\ge (1-\eps)(\|x\|+|t|).
 \eeq
\end{definition}

%%%%%%%%%%%%%%%%%%%%%%%%%%%%%%%%%%
\begin{definition} \label{2-def2}
Let $E$ be a \BS. A sequence
 $\{e_n\}_{n \in \N} \subset B_E \setminus \{0\}$ is said to be
an {\it asymptotic $\ell_{1}$-sequence}\/ if there is a sequence of
numbers
$\eps_n > 0$ with $\prod_{n \in \N}(1 - \eps_n) > 0$ such that
$e_{n+1}$ is $(\eps_n, 1)$-orthogonal to
$Y_n:=\lin\{e_{1},\ldots,e_{n}\}$ for every $n \in \N$.
\end{definition}

Evidently every  asymptotic $\ell_{1}$-sequence is  $1 / \prod_{n
\in \N}(1 - \eps_n)$-equivalent to the unit vector basis in
$\ell_{1}$, and moreover every element of the unit sphere of $E_m
:= \lin \{e_k\}_{k=m+1}^\infty$ is $\bigl(1 - \prod_{n\ge m}(1 -
\eps_n),1 \bigr)$-orthogonal to $Y_m$ for every $m \in \N$.
%%%%%%%%%%%%%%%%%%%%%%%%%%%%%%%%%%

The following lemma is completely analogous
to \cite[Lemma~2.8]{KadSSW}; instead of \cite[Lemma~2.1]{KadSSW} it uses
(v) of Lemma~\ref{1-th2}.  So we state it without proof.

%%%%%%%%%%%%%%%%%%%%%%%%%%%%%%%%%%

\begin{lemma}\label{2-L2}
Let $Y$ be a norming subspace of $X^*$, $X \in \DPr(Y)$, and let
 $Y_{0} \subset Y$ be a finite-dimensional subspace. Then for every
$\eps_{0}>0$ and every weak$^*$ slice $S(x_{0},\eps_{0})$ of
$B_{X^*}$ there is another weak$^*$ slice $S(x_{1},\eps_{1})
\subset S(x_{0},\eps_{0})$ of $B_{X^*}$ such that every element
$e^* \in S(x_{1},\eps_{1})$ is $(\eps_0, 1)$-orthogonal to
$Y_{0}$. In particular there is an element $e_1^* \in
S(x_{0},\eps_{0}) \cap S_Y$ which is $(\eps_0, 1)$-orthogonal to
$Y_{0}$.
\end{lemma}

%%%%%%%%%%%%%%%%%%%%%%%%%%%%%%%%%%

We need one more definition.

\begin{definition} \label{2-def3}
A sequence $\{e_n^*\}_{n \in \N} \subset B_{X^*}$ is said to be
{\it double-norming}\/ if $\lin \{e_k^*\}_{k=n}^\infty$ is norming
for every $n \in \N$.
\end{definition}

%%%%%%%%%%%%%%%%%%%%%%%%%%%%%%%%%%

Here is  the main result of this section.

\begin{theorem}\label{2-T1}
A separable \BS\ $X$ is an almost Daugavet space if and only if $X^*$
contains a double-norming asymptotic $\ell_{1}$-sequence.
\end{theorem}
%%%%%%%%%%%%%%%%%%%%%%%%%%%%%%%%%%

\begin{proof}
First we prove the ``if" part. Let $\{e_n^*\}_{n \in \N} \subset
B_{X^*}$ be a double-norming asymptotic $\ell_{1}$-sequence, and
let $\eps_n > 0$ with $\prod_{n \in \N}(1 - \eps_n) > 0$ be such that
$e_{n+1}^*$ is $(\eps_n, 1)$-orthogonal to $Y_n :=
\lin\{e_{1}^*,\ldots,e_{n}^*\}$ for every $n \in \N$. Let us prove
that $X$ has the \DP\ with respect to $Y =
\overline\lin\{e_n^*\}_{n \in \N}$ where the closure is meant in
the norm topology. To do this let us apply (iv) of
Lemma~\ref{1-th2}.

Fix an $x^*\in S_{Y}$, an $\eps > 0$ and a  weak$^*$ slice
$S(x,\eps)$ of the dual ball $B_{X^{*}}$. Denote in addition to
$Y_m=\lin\{e_{1}^*,\ldots,e_{m}^*\}$, $E_m := \lin
\{e_k^*\}_{k=m+1}^\infty$.  Using the definition of $Y$ select  an
$m \in \N$ and  an $x_m^*\in Y_m$ such that $\|x^*- x_m^*\| <
\eps/2$ and $\prod_{n \ge m }(1 - \eps_n) > 1 - \eps/2$. Since
$E_m$ is norming, there is a $y^*\in S(x,\eps) \cap S_{E_m}$.
Taking into account that every element of the unit sphere of $E_m$
is $(\eps/2,1)$-orthogonal to $Y_m$ we obtain
 $$
 \|x^{*}+ y^*\|\ge  \|x_m^{*}+ y^*\| - \|x^*- x_m^*\| \ge  2-\eps.
 $$

For the ``only if" part we proceed as follows. First we fix a
sequence of numbers $\eps_n > 0$ with $\prod_{n \in \N}(1 -
\eps_n) > 0$ and a dense sequence $(x_n)$ in $S_X$. We can choose
these $x_n$ in such a way that each of them appears in the
sequence $(x_n)$ infinitely many times. Assume now that $X \in
\DPr(Y)$ with respect to a norming subspace $Y \subset X^*$.
Starting with $Y_0=\{0\}$, $\eps_{0} = 1$ and applying
Lemma~\ref{2-L2} step-by-step we can construct  a sequence
$\{e_n^*\}_{n \in \N} \subset S_{Y}$ in such a way that each
$e_{n+1}^*$ belongs to $S(x_{n},\eps_{n})$ and is $(\eps_{n},
1)$-orthogonal to $Y_{n}$, where
$Y_n=\lin\{e_{1}^*,\ldots,e_{n}^*\}$ as before. This inductive
construction ensures that the $e_n^*$, $n \in \N$ form an
asymptotic $\ell_1$-sequence. On the other hand this sequence
meets
 every slice $S(x_{n},\eps_{n})$ infinitely many times, and this
 implies by density of $(x_n)$ that $(e_n^*)$ is double-norming.
\end{proof}

In Corollary~\ref{2-T1-cor}
we shall observe a somewhat more pleasing version of the last result.

We conclude the section with two examples.

\begin{prop}\label{2-T4}
The real space $\ell_{1}$ is an almost Daugavet space.
\end{prop}

\begin{proof}
To prove this statement we will construct a double-norming
asymptotic $\ell_{1}$-sequence $(f_n) \subset
\ell_{\iy}=(\ell_{1})^*$. At first consider a sequence $(g_n)
\subset \ell_{\iy}$ of elements $g_n = (g_{n,j})_{j \in \N}$ with
all $g_{n,j} = \pm 1$ satisfying the following independence
condition: for arbitrary finite collections $\alpha_s = \pm 1$, $s
= 1, \ldots, n$, the set of those $j$ that $g_{s,j} = \alpha_s$
for all $s = 1, \ldots, n$ is infinite (for instance, put
$g_{s,j}:= r_s(t_j)$, where the $r_s$ are the Rademacher functions
and  $(t_j)_{j \in \N}$ is a fixed  sequence of irrationals that
is dense in $[0,1]$). These $g_n$, $n \in \N$, form an isometric
$\ell_{1}$-sequence, and moreover if one changes a finite number
of coordinates in each of the $g_n$ to some other $\pm 1$, the
independence condition will survive, so the modified sequence will
still be an isometric $\ell_{1}$-sequence.

Now let us define the vectors $f_n = (f_{n,j})_{j \in \N}$,
$f_{n,j} = \pm 1 $, in such a way that for
$k=1,2,\ldots$ and
$n = {2^k + 1}, \allowbreak {2^k + 2}, \allowbreak \ldots, 2^{k+1}$
the vectors $(f_{n,j})_{j=1}^k \in \ell_{\iy}^{(k)}$
run over all extreme points of the unit ball of
$\ell_{\iy}^{(k)}$, i.e., over all possible $k$-tuples of $\pm 1$;
for the remaining values of indices we put $f_{n,j} = g_{n,j}$.
As we have already remarked, the $f_n$ form an isometric
$\ell_{1}$-sequence. Moreover, for every $k \in \N$ the
restrictions of the $f_n$ to the first $k$ coordinates form a
double-norming sequence over $\ell_{1}^{(k)}$, so $(f_n)_{n \in
\N}$ is a double-norming sequence over $\ell_{1}$.
\end{proof}

Some ideas of the previous proof will enter into the proof of
Theorem~\ref{th0.1}. As a consequence of that theorem, the complex
space $\ell_1$ is almost Daugavet as well. It is worth noting that
$\ell_1$ fails the \DP\ and cannot even be renormed to have it 
(see e.g.\ \cite[Cor.~2.7]{KadSSW}). 

Since $\ell_\infty$ is isomorphic to $L_\infty[0,1]$, which has
the \DP, $\ell_\infty$  can be equivalently renormed to possess
the \DP. Let us show that in the original norm it is not even an
almost Daugavet space.
This is a special case of the following proposition in which $\K$
stands for $\R$ or~$\C$.

\begin{prop}\label{2-T5}
No Banach space of the form $Z=X\oplus_\infty \K$ is an almost
Daugavet space. 
\end{prop}

\begin{proof}
Let us call a functional $z_0^*\in Z^*$ a \textit{Daugavet
functional} if
$$
\| \Id + z_0^*\otimes z_0\| = 1 + \|z_0^*\otimes z_0\| \qquad \mbox{for
every }  z_0\in Z.
$$
We shall show that $z_0^*= (x_0^*, b_0)$ is not a Daugavet functional
if $b_0\neq 0$. Hence all the Daugavet functionals lie in the weak$^*$
closed subspace $(\{0\} \oplus X)^\bot$ of~$Z^*$.

So let $x_0^*\in X^*$ and $b_0\neq0$ with $\|x_0^*\|+|b_0|=1$,
$z_0^*=(x_0^*,b_0)$ and let $z_0= (0, -{|b_0|}/b_0)$. If
$z=(x,a)\in B_Z$, i.e., $\|x\|\le 1$ and $|a|\le1$, then
\begin{eqnarray*}
\|z+z_0^*(z)z_0\| &=&
\max\{ \|x\|, |\,a-z_0^*(z){|b_0|}/b_0\,| \} \\
&\le&
\max\{1, |\,a-(x_0^*(x_0)+b_0a){|b_0|}/b_0\,| \} \\
&\le&
\max\{1, \|x_0^*\| + (1-|b_0|)\} <2.
\end{eqnarray*}
This shows that $z_0^*$ is not a Daugavet functional. 
\end{proof}

If $K$ is a compact Hausdorff space with an isolated point, then
$C(K)$ is of the form $X\oplus_\infty \K$, hence it fails the almost
\DP. But if $K$ is an uncountable metric space, then $C(K)$ is
isomorphic to $C[0,1]$ by Milutin's theorem \cite[Th.~III.D.19]{Woj}, 
hence it can be renormed to have the \DP.

%%%%%%%%%%%%%%%%%%%%%%%%%%%%%%%%%%

\section{Proof of Theorem \ref{th0.1}}

Since the three properties considered in Theorem~\ref{th0.1} hold
for a complex Banach space $X$ if and only if they hold for the
underlying real space $X_\R$, we will tacitly assume in this section
that we are dealing with real spaces.

We will accomplish the proof of Theorem~\ref{th0.1} by means of the
following propositions.

The following fact applied for separable spaces is equivalent to
implication (c) $\Rightarrow$ (a) of Theorem~\ref{th0.1}.

\begin{prop}\label{2-T3}
Every almost Daugavet space $X$ has thickness $T(X)=2$.
\end{prop}

\begin{proof}
Let $Y \subset X^*$ be a norming subspace with respect to which
$X \in \DPr(Y)$. According to the definition of $T(X)$ we have to
show that for every $\eps_0 > 0$ there is no finite
$(2-\eps_0)$-net of $S_X$ consisting of elements of $S_X$. In other
words we must demonstrate that for every collection $\{x_1,
\ldots, x_n \} \subset S_X$ there is a $y_0 \in S_X$ with $\|x_k -
y_0\| > 2-\eps_0$ for all $k = 1, \ldots, n$. But this is an
evident corollary of Lemma~\ref{1-th2}(iii): starting with an
arbitrary $y_0^* \in S_{Y^*}$ and applying (iii) we can construct
recursively elements $y_k^* \in S_{Y^*}$ and reals $\eps_k \in (0,
\eps)$, $k = 1, \ldots, n$, in such a way that $S(y_{k}^*,\eps_k)
\subset S(y_{k-1}^*,\eps_{k-1})$ and
$$
\|(-x_k) + y\| > 2-\eps_0
$$
 for every $y \in S(y_{k}^*,\eps_k)$. Since $S(y_{n}^*,\eps_n)$
 is the smallest of the slices constructed, every norm-one element
 of $S(y_{n}^*,\eps_n)$ can be taken as the $y_0$ we need.
\end{proof}

For spaces with the \DP\ the previous proposition
has been proved in  \cite[Prop.~4.1.6]{Barreno}.

Let us  now turn to the implication (a) $\Rightarrow$ (b) of
Theorem~\ref{th0.1}.

\begin{prop}\label{a-b}
If $T(X)=2$ and $X$ is separable, then $X$ contains a canonical
$\ell_1$-type sequence.
\end{prop}

\begin{proof}
Fix a dense countable set $A=\{a_n\dopu n \in \N\}\subset S_X$ and a
null-sequence $(\eps_n)$ of positive reals. Since for every $n \in
\N$ the $n$-point set $\{-a_1, \ldots, -a_n\}$ is not a $(2 -
\eps_n)$-net of $S_X$ there is an $e_n \in S_X$ with $\|e_n -
(-a_k)\| > 2 - \eps_n$  for all  $k = 1, \ldots, n$. The
constructed sequence $(e_n)$ satisfies for every $k \in \N$ the
condition
$$
\lim_{n\to\infty} \|a_k+e_n\| =\|a_k\|+1 = 2.
$$
By the density of $A$ in $S_X$ and a standard convexity argument
(cf.\ e.g.\ \cite[page~78]{Wer})
this yields that $(e_n)$ is a canonical $\ell_1$-type sequence.
\end{proof}

By the result in \cite{BCP} mentioned in the introduction we
obtain:

\begin{cor}\label{ell1cor}
Every almost Daugavet space contains $\ell_1$. 
\end{cor}

%%%%%%%%%%%%%%%%%%%%%%%%%%%%%%%%%%
%%%%%%%%%%%%%%%%%%%%%%%%%%%%%%%%%%
It remains to prove the  implication (b) $\Rightarrow$ (c) of
Theorem~\ref{th0.1}.

\begin{prop}\label{theo08-1}
A separable Banach space $X$ containing a canonical  $\ell_1$-type
sequence is an almost Daugavet space.
\end{prop}

\begin{proof}
We will use Theorem~\ref{2-T1}.
Fix an increasing sequence of finite-dim\-ensional subspaces $E_1
\subset E_2 \subset E_3 \subset \ldots$ whose union is dense in~$X$. Also, fix
sequences $\eps_n\searrow 0$ and $\delta_n>0$ such that for all~$n$
\begin{equation}\label{eq08-01}
  \prod_{k=n}^\infty (1-\delta_k) \ge 1-\eps_n .
\end{equation}
Passing to a subsequence if necessary  we can find a canonical
 $\ell_1$-type sequence  $(e_n)$ satisfying the following additional
condition: For every $x\in \lin (E_n\cup\{e_1,\dots,e_n\})$ and every
$\alpha\in \R$ we have
\begin{equation}\label{eq08-02}
  \|x+\alpha e_{n+1}\| \ge (1-\delta_n)(\|x\|+|\alpha|).
\end{equation}
Then we have for every $x\in E_n$ and every $y= \sum_{k=n+1}^M a_k
e_k$ by (\ref{eq08-01}) and (\ref{eq08-02})
\begin{equation}\label{eq08-03}
\|x+y\| \ge (1-\eps_n) \|x\| + \sum_{k=n+1}^M (1-\eps_{k-1}) |a_k|.
\end{equation}
Fix a dense sequence $(x_n)$ in $S_X$ such that $x_n\in E_n$ and every
element of the range of the sequence is attained infinitely often,
that is for each $m\in\N$ the set $\{n\dopu x_n=x_m\}$ is infinite.
Finally, fix an ``independent'' sequence $(g_n)\subset \ell_ \infty$,
$g_{n,j}=\pm1$, as  in the proof of Proposition~\ref{2-T4}.

Now we are ready to construct a double-norming asymptotic
$\ell_1$-sequence $(f_n^*)\subset X^*$. First we define $\tilde f_n^*$
on $F_n:= \lin\{x_n,e_{n+1}, e_{n+2}, \dots\}$ by
\begin{eqnarray}
\label{eq08-04}
  \tilde f_n^*(x_n) &=& 1-\eps_n, \\
\label{eq08-05}
\tilde f_n^*(e_k) &=& (1-\eps_{k-1})g_{n,k} \qquad\mbox{(if }k>n).
\end{eqnarray}
By (\ref{eq08-03}), $\|\tilde f_n^*\|\le 1$, and indeed $\|\tilde
f_n^*\|= 1$ by (\ref{eq08-05}). Define $f_n^*\in X^*$ to be a
Hahn-Banach extension of $\tilde f_n^*$. Condition~(\ref{eq08-04})
and the choice of $(x_n)$ ensure that $(f_n^*)$ is double-norming.
Let us show that it is an isometric $\ell_1$-basis. Indeed, due to
our definition of an ``independent'' sequence, for an arbitrary finite
collection $A =\{a_1, \ldots, a_n \}$ of non-zero coefficients the
set $J_A$ of those $j > n$ that $g_{s,j} = \sign a_s$, $s = 1,
\ldots, n$, is infinite. So by (\ref{eq08-05})
$$
\Biggl\|\sum_{s=1}^n a_s f_s^* \Biggr\| \ge \sup_{j \in J_A} \Biggl(
\sum_{s=1}^n a_s f_s^*\Biggr)e_j = \sup_{j \in J_A} (1-\eps_{j-1})
\sum_{s=1}^n |a_s| = \sum_{s=1}^n |a_s|.
$$
\end{proof}

Since we have constructed an isometric $\ell_1$-basis 
(over the reals) in the last
proof, we have obtained the following version of Theorem~\ref{2-T1}.

\begin{cor}\label{2-T1-cor}
A real separable \BS\ $X$ is an almost Daugavet space 
if and only if $X^*$
contains a double-norming isometric $\ell_{1}$-sequence.
\end{cor}
%%%%%%%%%%%%%%%%%%%%%%%%%%%%%%%%%
\section{Proof of Theorem \ref{th0.2}}

We start with two lemmas.

\begin{lemma}\label{lem08-1}
Let $X$ be a linear space, $(e_n)\subset X$, and let $p$ be a seminorm
on $X$. Assume that $(e_n)$ is an isometric $\ell_1$-basis with
respect to $p$, i.e.,
$p(\sum_{k=1}^n a_k e_k )= \sum_{k=1}^n |a_k|$ for all
$a_1,a_2,\ldots\in\K$. 
Fix a free ultrafilter $\mathcal{U}$ on $\N$ and define
$$
p_r(x)= \mathcal{U}\mbox{-}\lim_n p(x+r e_n)-r
$$
for $x\in X$ and $r>0$. Then:
 \begsta
  \item $0\le p_r(x)\le p(x)$
for all $x\in X$,
 \item $p_r(x)=p(x)$ for all $x\in
\lin\{e_1,e_2,\dots\}$,
 \item the map $x\mapsto p_r(x)$ is convex
for each $r$,
 \item the map  $r\mapsto p_r(x)$ is convex for each $x$,
\item $p_r(tx)= tp_{r/t}(x)$ for each $t>0$,
\item $|p_r(x) - p_r(y)| \le p(x-y)$ for all $x,y\in X$.
\endsta
\end{lemma}

\begin{proof}
The only thing that is not obvious is that $p_r$ is positive; note
that (b) is a well-known property of the unit vector basis of
$\ell_1$. Now, given $\eps>0$ pick $n_\eps$ such that
$$
p(x+re_{n_\eps}) \le \mathcal{U}\mbox{-}\lim_n p(x+r e_n) + \eps.
$$
Then for each $n\neq n_\eps$
\bea
p(x+re_n) &=& p(x+re_{n_\eps} +r (e_n-e_{n_\eps})) \\
&\ge&
2r - p(x+re_{n_\eps}) \\
& \ge&
2r - \mathcal{U}\mbox{-}\lim_n p(x+r e_n) - \eps;
\eea
hence
$2\mathcal{U}\mbox{-}\lim_n p(x+re_n) \ge  2r-2\eps$ and $p_r\ge0$.
\end{proof}

%%%%%%%%%%%%%%%%%%%%%%%%%%%%%%%%%

\begin{lemma}\label{lem08-2}
Assume the conditions of Lemma~\ref{lem08-1}. Then the function
$r\mapsto p_r(x)$ is decreasing  for each $x$. The quantity
$$
\bar p (x):= \lim_{r\to\infty}  p_r(x) = \inf_r p_r(x)
$$
satisfies \textrm{(a)--(c)} of Lemma~\ref{lem08-1} and moreover
\begin{equation}\label{eq08-06}
  \bar p(tx) = t \bar p(x) \qquad\mbox{for }t>0,\ x\in X.
\end{equation}
\end{lemma}

\begin{proof}
By Lemma~\ref{lem08-1}(a) and (d),
$r\mapsto p_r(x)$ is bounded and convex, hence
decreasing. Therefore, $\bar p$ is well defined. Clearly,
(\ref{eq08-06}) follows from (e) above.
\end{proof}

%%%%%%%%%%%%%%%%%%%%%%%%%%%%%%%%%
Since for separable spaces the condition $T(X)=2$ is equivalent to
the presence of a  canonical $\ell_1$-type sequence  and a  canonical
 $\ell_1$-type sequence evidently contains a subsequence
equivalent to the canonical basis of $\ell_1$, to prove
Theorem~\ref{th0.2} it is sufficient to demonstrate the following:

\begin{theorem}\label{theo08-2}
Let $X$ be a Banach space containing a copy of $\ell_1$. Then $X$
can be renormed to admit a canonical $\ell_1$-type sequence.
Moreover if $(e_n) \subset X$ is an arbitrary sequence equivalent
to the canonical basis of $\ell_1$ in the original norm, then one
can construct an equivalent norm on $X$ in such a way that $(e_n)$
is isometrically equivalent to the canonical basis of $\ell_1$ and
$(e_n)$ forms a canonical $\ell_1$-type sequence in the new norm.
\end{theorem}

\begin{proof}
Let $Y$ be a subspace of $X$ isomorphic to $\ell_1$, and let $(e_n)$
be its canonical basis. To begin with, we can renorm $X$ in such a way
that $Y$ is isometric to $\ell_1$ and  $(e_n)$ is an isometric
$\ell_1$-basis.

Let $\mathcal{P}$ be the family of all
seminorms $\tilde p$ on $X$ that are dominated
by the norm of $X$ and for which $\tilde p(y)= \|y\|$ for $y\in Y$.
By Zorn's lemma, $\mathcal{P}$ contains a minimal element, say~$p$.
We shall argue  that
\begin{equation}\label{eq08-07}
  \lim_{n\to \infty} p(x+e_n) = p(x)+1 \qquad \forall x\in X.
\end{equation}
To show this it is sufficient to prove  that for every free ultrafilter
$\mathcal{U}$ on $\N$
\begin{equation}\label{eq08-08}
  \mathcal{U}\mbox{-}\lim_{n} p(x+e_n) = p(x)+1 \qquad \forall x\in X.
\end{equation}
To this end associate to $p$ and $\mathcal{U}$ the functional $\bar p$ from
Lemma~\ref{lem08-2}. Note that $0\le \bar p\le p$, but
a priori $\bar p$ need not be a seminorm. However,
$$
q(x) = \frac{ \bar p(x) + \bar p(-x)}2
$$
in the real case, resp.\ 
$$
q(x)= \int_0^1 \bar p(e^{2\pi it}x)\,dt
$$
in the complex case, 
defines a seminorm, and $q\le p$.
(Lemma~\ref{lem08-1}(f) implies that the integrand is a
continuous function of~$t$.)
By Lemma~\ref{lem08-1}(b) and by
minimality of $p$ we get that
\begin{equation}\label{eq08-09}
  q(x)=p(x) \qquad \forall x\in X.
\end{equation}
Now, since  $p(x)=p(\lambda x)\ge \bar p(\lambda x)$ whenever
$\lambda$ is a scalar of modulus~1,
%Now, since $p(x)\ge \bar p(x)$ and $p(x)=p(-x)\ge \bar p(-x)$,
(\ref{eq08-09}) implies that $p(x)=\bar p(x)$.
Finally, by Lemma~\ref{lem08-1}(a) and the definition of $\bar p$ we
have $p(x)=p_r(x)$ for all $r>0$; in particular $p(x)=p_1(x)$, which
is our claim~(\ref{eq08-08}).

To complete the proof of the theorem, consider
$$
\Norm x \Norm := p(x) + \|[x]\|_{X/Y};
$$
this is the equivalent norm that we need. Indeed, clearly $\Norm
x\Norm \le 2\|x\|$. On the other hand, $\Norm
x\Norm \ge \frac13 \|x\|$. To see this assume $\|x\|=1$. If
$\|[x]\|_{X/Y}\ge \frac13$, there is nothing to prove. If not, pick
$y\in Y$ such that $\|x-y\|<\frac13$. Then
$p(y)=\|y\|>\frac23$, and
$$
\Norm x \Norm \ge p(x) \ge p(y)- p(x-y) > \frac23 - \|x-y\| > \frac13.
$$
Therefore, $\|\;.\;\|$ and $\Norm\;.\;\Norm$ are equivalent norms, and
$$
\lim_{n\to\infty} \Norm x+e_n \Norm =
\lim_{n\to\infty} p( x+e_n ) + \|[x]\|_{X/Y} =
p(x) + 1 + \|[x]\|_{X/Y} =
\Norm x \Norm + 1
$$
shows that $(e_n)$ is a  canonical $\ell_1$-type sequence  for the
new norm.
\end{proof}

We would like to mention another proof of Theorem~\ref{theo08-2} that
was suggested to us by W.B. Johnson. In this proof $X$ is a real
Banach space. Let again $Y\subset X$ be a subspace isometric to
$\ell_1$ with canonical basis $(e_n)$. We denote by $(r_n)$ the sequence
of Rademacher functions in $L_\infty[0,1]$. Then there is a norm-1
operator from $Y$ to $L_\infty[0,1]$ mapping  $e_n$ to $r_n$, for
each~$n$. Since $L_\infty[0,1]$ is 1-injective, the operator can be
extended to a norm-1 operator $T\dopu X\to L_\infty[0,1]$. If we let
$$
\Norm x\Norm = \|Tx\| + \|[x]\|_{X/Y},
$$
then this equivalent norm works; the details of the computation are
the same as above.

%%%%%%%%%%%%%%%%%%%%%%%%%%%%%%%%%%%%%%%%%%%%%%%%%%%%%%%%%%%%%%%%%%%%%
%%%%%%%%%%%%%%%%%%%%%%%%%%%%%%%%%%%%%%%%%%%%%%%%%%%%%%%%%%%%%%%%%%%%%

\end{document}

%%%%%%%%%%%%%%%%%%%%%%%%%%%%%%%%%%%%%%%%%%%%%%%%%%%%%%%%%%%%%%%%%%%%%
%%%%%%%%%%%%%%%%%%%%%%%%%%%%%%%%%%%%%%%%%%%%%%%%%%%%%%%%%%%%%%%%%%%%%